\documentclass[graybox]{svmult}

\usepackage{mathptmx}       
\usepackage{helvet}         
\usepackage{courier}        
\usepackage{type1cm}        
\usepackage{makeidx}         
\usepackage{graphicx}        
\usepackage{multicol}        
\usepackage[bottom]{footmisc}

\usepackage{amsbsy}
\usepackage{amsfonts}
\usepackage{amssymb}
\usepackage{latexsym}

\newcommand{\boldgreek}[1]{\mbox{\boldmath$#1$}}

\newcommand{\R}{I\kern-0.37emR}

\newcommand{\ny}{n\rightarrow\infty}
\newcommand{\Q}{I\kern-0.37emP}

\makeindex             

\begin{document}
\title*{Affine equivariant {rank-weighted L-estimation} of multivariate location}
\author{Pranab Kumar Sen,~ Jana Jure\v{c}kov\'a~ and Jan Picek}
\institute{Pranab Kumar Sen \at Departments of Statistics and Biostatistics, University of North Carolina at Chapel Hill, USA,\\ \email{pksen@bios.unc.edu}
\and Jana Jure\v{c}kov\'a \at Department of Probability and Statistics, Faculty of Mathematics and Physics, Charles University, CZ-186 75 Prague 8, Czech Republic, \email{jurecko@karlin.mff.cuni.cz}
\and Jan Picek \at Department of Applied Mathematics, Technical University of Liberec, Czech Republic, \email{jan.picek@tul.cz}}
%
%
\maketitle

\abstract{{In the multivariate one-sample location model, we propose a class of fle\-xi\-ble robust, affine-equivariant L-estimators of location, for distributions invoking affine-invariance of Mahalanobis distances of individual observations. An involved iteration process for their computation is numerically illustrated.}} 

\section{Introduction}
\label{sec:1}

The 
affine-equivariance and its dual  affine-invariance are natural generalizations of univariate translation-scale equivariance and invariance notions (Eaton 1983). 
Consider the group $\mathcal C$ of transformations of $\mathbb R_p$ to $\mathbb R_p$: 
\begin{equation}\label{8.2.2} 
\mathbf X\mapsto\mathbf Y=\mathbf B\mathbf X+\mathbf b, \; \mathbf b\in\mathbb R_p
\end{equation}
where $\mathbf B$ is a positive definite $p\times p$ matrix. Generally with the choice of $\mathbf B$ we do  not transform dependent coordinates of $\mathbf X$ to stochastically independent coordinates of $\mathbf Y.$ This is possible when $\mathbf X$ has  a multi-normal distribution with mean vector $\boldgreek\theta$ and positive definite dispersion matrix $\boldgreek\Sigma,$
when letting  $\boldgreek\Sigma^{-1}=\mathbf B\mathbf B^{\top},$ so that $\mathbb E\mathbf Y=\boldgreek\xi =\mathbf B\boldgreek\theta+\mathbf b$ and dispersion matrix
$\mathbf B\boldgreek\Sigma\mathbf B^{\top}=\mathbf I_p.$ 
{To construct and study the affine equivariant estimator of the location $\boldgreek\theta$  
we need to consider some affine-invariant (AI) norm.}
The most well-known affine invariant norm  
is the Mahalanobis norm, whose squared version is  
\begin{equation}\label{8.2.3}
 \boldgreek\Delta^2=(\mathbf X- \boldgreek\theta)^{\top}\boldgreek\Sigma^{-1}(\mathbf X-\boldgreek\theta)=\|\mathbf X-\boldgreek\theta\|^2_{\Sigma}
\end{equation}
where $\boldgreek\Sigma$ is the dispersion matrix of $\mathbf X.$ 
{To incorporate this norm, we need to use its empirical version based on independent sample $\mathbf X_1,\ldots,\mathbf X_n,$ 
namely
$$s_{ij}=\frac 12(\mathbf X_i-\mathbf X_j)^{\top}\mathbf V_n^{*-1}(\mathbf X_i-\mathbf X_j), \; 1\leq i<j\leq n$$
where $\mathbf V_n^*=(n(n-1))^{-1}\sum_{1\leq i<j\leq n}(\mathbf X_i-\mathbf X_j)(\mathbf X_i-\mathbf X_j)^{\top}.$ To avoid redundancy, we may consider the reduced set 
\begin{eqnarray}\label{8.5.11}
&&\tilde{d}_{ni}=(\mathbf X_i-\mathbf X_n)^{\top}\tilde{\mathbf V}_n^{-1}(\mathbf X_i-\mathbf X_n), \; i=1,\ldots,n-1\\
&&\tilde{\mathbf V}_n=\sum_{i=1}^{n-1}(\mathbf X_i-\mathbf X_n)(\mathbf X_i-\mathbf X_n)^{\top}
\end{eqnarray}
 which forms the maximal invariant (MI) 
 with respect to affine transformations  (\ref{8.2.2}).
 } 
{An equivalent form of the maximal invariant is
\begin{eqnarray}\label{8.5.11b}
&&{d}_{ni}=(\mathbf X_i-\overline{\mathbf X}_n)^{\top}{\mathbf V}_n^{-1}(\mathbf X_i-\overline{\mathbf X}_n), \; i=1,\ldots,n\\
&&{\mathbf V}_n=\sum_{i=1}^{n}(\mathbf X_i-\overline{\mathbf X}_n)(\mathbf X_i-\overline{\mathbf X}_n)^{\top}
\end{eqnarray}
(Obenchain (1971)).}
{Note that all the $d_{ni}$ are between 0 and 1 and their sum equals to $p.$ Because $d_{ni}$ are exchangeable, bounded random variables, all nonnegative, with a constant sum equal to $p,$  the asymptotic properties of the array $(d_{n1},\ldots,d_{nn})^{\top}$ follow from Chernoff and Teicher (1958) and Weber (1980). Similarly, $\sum_{i=1}^{n-1} \tilde{d}_{ni}=p.$ 
Neither $\overline{\mathbf X}_n$ nor $\mathbf V_n$ is robust against outliers and gross errors contamination. As such, we are motivated to replace $\overline{X}_n$ and $\mathbf V_n$ by suitable robust versions and incorporate them in the formulation of a robust affine equivariant estimator of $\boldgreek\theta.$ If $\widehat{\boldgreek\theta}_n$ is some affine equivariant estimator of $\boldgreek\theta,$ then writing 
$$\widehat{\mathbf V}_n=
\sum_{i=1}^n(\mathbf X_i-\widehat{\boldgreek\theta}_n)(\mathbf X_i-\widehat{\boldgreek\theta}_n)^{\top}$$ 
we may note that $\widehat{\mathbf V}_n$ is smaller than $\mathbf V_n$ in the matrix sense. However, it cannot be claimed that the  Mahalanobis distances (\ref{8.5.11b}) can be made shorter by using 
 $\widehat{\boldgreek\theta}_n$ instead of $\overline{\mathbf X}_n,$ because 
$\sum_{i=1}^n(\mathbf X_i-\widehat{\boldgreek\theta}_n)^{\top}\widehat{\mathbf V}_n^{-1}(\mathbf X_i-\widehat{\boldgreek\theta}_n)=p.$
 Our motivation is to employ the robust Mahalanobis distances in the formulation of robust affine equivariant estimator of $\boldgreek\theta,$ through a tricky ranking of the Mahalanobis distances in (\ref{8.5.11b}) and an iterative procedure in updating an affine equivariant robust estimator of $\boldgreek\theta.$
 
The robust estimators in the multivariate case, discussed in detail in Jure\v{c}kov\'a et al. (2013), are not automatically affine equivariant. With due emphasize on the spatial median, some other estimators were considered by a host of researchers, and we refer to Oja (2010) and Serfling (2010) for a detailed account. Their emphasis is on the spatial median and spatial quantile functions defined as follows:}\\
 Let $\mathcal B_{p-1}(\mathbf 0)$ be the open unit ball. Then the $\mathbf u$-th \textit{spatial quantile} $Q_F(\mathbf u), \; \mathbf u\in \mathcal B_{p-1}(\mathbf 0)$ is defined as the  solution ${\mathbf x}=Q_F(\mathbf u)$ 
of the equation
$$\mathbf u=\mathbb E\left\{\frac{{\mathbf x}-\mathbf X}{\|{\mathbf x}-\mathbf X\|}\right\}, \quad \mathbf u\in\mathcal B_{p-1}(\mathbf 0).$$
Particularly, $Q_F(\mathbf 0)$ is the spatial median. It is equivariant with respect to $\mathbf y =\mathbf B\mathbf x+\mathbf b, \; \mathbf b\in\mathbb R_p, \; \mathbf B$ positive definite and \textit{orthogonal}. However, the spatial quantile function may not be affine-equivariant for all $\mathbf u.$

Among various approaches to multivariate quantiles  we refer to Chakraborty (2001), Roelant and van Aelst (2007), Hallin et al. (2010), Kong and Mizera (2012), Jure\v{c}kov\'a et al. (2013), among others. Lopuha$\ddot{\rm a}$ \& Rousseeuw (1991) and Zuo (2003, 2004 2006), Lopuha$\ddot{\rm a}$ \& Rousseeuw (1991) and Zuo (2003, 2004 2006), among others, studied robust affine-equivariant estimators with high breakdown point, {based on projection debths}. {An alternative approach based on the notion of the depth function and associated U-statistics has been initiated by Liu at al. (1999). Notice that every affine-invariant function of $(\mathbf{X}_1,\ldots,\mathbf{X}_n)$ depends on the $\mathbf X_i$ only through the maximal invariant; particularly, this applies to the ranks of the $d_{ni}$ and also to all affine invariant depths considered in the literature. 
In our formulation, affine equivariance property is highlighted and accomplished through a ranking of the Mahalanobis distances at various steps.}

\section{Affine equivariant linear estimators}
Let $\mathbf X\in\mathbb R_p$ be a random vector with a distribution function $F.$ {Unless stated other\-wi\-se, we
assume that $F$ is absolutely continuous.} 
Consider the qroup $\mathcal C$ of affine transformations $\mathbf X\mapsto\mathbf Y=\mathbf B\mathbf X+\mathbf b$ with $\mathbf B$
nonsingular of order $p\times p,$ $\mathbf b\in\mathbb R_p.$ Each transformation generates a distribution function $G$ also defined on $\mathbb R_p,$ which we denote $G=F_{B,b}.$
 A vector-valued functional $\boldgreek\theta(F),$ designated as a suitable measure of location of $F,$  is said to be an \textit{affine-equivariant location functional}, provided
$$\boldgreek\theta(F_{B,b})= \mathbf B\boldgreek\theta(F)+\mathbf b \quad \forall\mathbf b\in\mathbb R_p, \; \mathbf B \; \mbox{ positive definite}.$$
Let $\boldgreek\Gamma(F)$ be a matrix valued functional of $F,$ designated as a measure of the \textit{scatter} of $F$ around its location $\boldgreek\theta$ and capturing
its \textit{shape} in terms of variation and covariation of the coordinate variables. $\boldgreek\Gamma(F)$ is often termed a \textit{covariance functional}, and a natural requirement is that it is independent of $\boldgreek\theta(F).$ It is termed an \textit{affine-equivariant covariance functional}, provided
$$\boldgreek\Gamma(F_{B,b})=\mathbf B\boldgreek\Gamma(F)\mathbf B^{\top} \quad \forall\mathbf b\in\mathbb R_p, \; \mathbf B \; \mbox{ positive definite}.$$ 

We shall construct {a class of affine equivariant} L-estimators of location parameter, starting with initial affine-equivariant location estimator and scale functional, {and then iterating them to a higher robustness}. For simplicity we start with the sample mean vector $\overline{\mathbf X}_n=\frac 1n\sum_{i=1}^n\mathbf X_i$ and with the matrix $\boldgreek V_n =\mathbf A_n^{(0)}
=\sum_{i=1}^n(\mathbf X_i-\overline{\mathbf X}_n)(\mathbf X_i-\overline{\mathbf X}_n)^{\top}=n\widehat{\boldgreek\Sigma}_n, \; n>p,$ where $\widehat{\boldgreek\Sigma}_n$ is the sample covariance matrix.  
Let
$R_{ni}=\sum_{j=1}^nI[d_{nj}\leq d_{ni}]$ be the rank of $d_{ni}$ among $d_{n1},\ldots,d_{nn}, \; i=1,\ldots,n,$ and denote $\mathbf R_n=(R_{n1},\ldots,R_{nn})^{\top}$ the vector of ranks. Because $F$ is continuous, the probability of ties is 0, hence the ranks are well defined.  
Note that $d_{ni}$ and $R_{ni}$ are affine-invariant, $i=1,\ldots,n.$ Moreover, the $R_{ni}$ are invariant under any strictly monotone transformation of $d_{ni}, \; i=1,\ldots,n.$ Furthermore, each $\mathbf X_{i}$ is {trivially} affine-equivariant. We introduce the following {(Mahalanobis)} ordering of $\mathbf X_1,\ldots,\mathbf X_n:$
\begin{equation}\label{8.5.12a}
\mathbf X_i\prec \mathbf X_j \; \Leftrightarrow  \; d_{ni}<d_{nj}, \; i\neq j=1,\ldots,n.
\end{equation}
This {affine invariant ordering} leads to vector of order statistics $\mathbf X_{n:1}\prec\ldots\prec\mathbf X_{n:n}$ {of the sample 
$\mathbb X_n$.}   
{In the univariate case with the order statistics $X_{n:1}\leq\ldots\leq X_{n:n}$, we can consider the $k$-order \textit{rank weighted mean} (Sen (1964)) defined as
$$T_{nk} =\left(
\begin{array}{c}
n\\
2k+1\\
\end{array}
\right)^{-1}
\sum_{i=k+1}^{n-k}
\left(
\begin{array}{c}
i-1\\
k\\
\end{array}
\right)
\left(
\begin{array}{c}
n-1\\
k\\
\end{array}
\right)X_{n:i}, \quad k=0,\ldots, \left[\frac{n+1}{2}\right].
$$
For $k=0, \; T_{nk}=\overline{X}_n$ and $k=[(n+1)/2]$ leads to the median $\widetilde{X}_n.$ In the multivariate case, the ordering is induced by the non-negative $d_{ni}$, and the smallest $d_{ni}$ {corresponds to the smallest \textit{outlyingness} from the center}, or to \textit{the nearest neighborhood {of the center}}. Keeping that in mind, we can conceive by a sequence $\{k_n\}$ of nonnegative integers, such that 
$k_n$ is $\nearrow$ in $n,$ but $n^{-1/2}k_n$ is $\searrow$ in $n,$ and for fixed $k$ put
$$\mathbf L_{nk}=\left(
\begin{array}{c}
k_n\\
k\\
\end{array}
\right)^{-1}\sum_{i=1}^nI\left[R_{ni}\leq k_n\right]\left(\begin{array}{c}
k_n-R_{ni}\\
k-1\\
\end{array}
\right)\mathbf X_i.$$
$\mathbf L_{nk}$ is affine-equivariant, because the $d_{ni}$ are affine invariant and the $\mathbf X_i$ are {trivially} affine equivariant. Of particular interest is the case of $k=1,$ i.e.,
$$\mathbf L_{n1}=k_n^{-1}\sum_{i=1}^nI\left[R_{ni}\leq k_n\right]\mathbf X_i$$
representing a trimmed, rank-weighted, nearest neighbor (NN) affine-equivariant estimator of $\boldgreek\theta.$ In the case $k=2$ we have
$$\mathbf L_{n2}=\left(\begin{array}{c}
k_n\\2\\\end{array}\right)^{-1}
\sum_{i=1}^nI\left[R_{ni}\leq k_n\right](k_n-R_{ni})\mathbf X_i$$
}
{which can be rewritten as $\quad\mathbf L_{n2}=\sum_{i=1}^nw_{nR_{ni}}\mathbf X_i$
with the weight-function
$$w_{ni}=\left\{\begin{array}{lll}
\left(\begin{array}{c}
k_n\\
2\\
\end{array}\right)^{-1}(k_n-i)&\ldots & i=1,\ldots,k_n;\\[3mm]  
\quad 0 &\ldots& i>k_n.\\
\end{array}
\right.
$$ 
}
We see that $\mathbf L_n$ puts greater influence for $R_{ni}=1$ or 2, and $w_{nk_n}=0; \; w_{n1}=2/k_n.$ For $k\geq 3,$ even greater weights will be given to $R_{ni}=1$ or 2, etc. 
For large $n$ we can use the Poisson weights, following Chaubey and Sen (1996): 
\begin{eqnarray*}
&&w^0_{ni}=\left(1-\rm e^{-\lambda}\right)^{-1}\frac{\rm e^{-\lambda}\lambda^i}{i!}, \; \lambda<1, \; i=1,2,\ldots .
\end{eqnarray*}
A typical $\lambda$ is chosen somewhere in the middle of $[0,1].$ Then
$\mathbf L_n^0=\sum_{i=1}^nw_{nR_{ni}}^0\mathbf X_i$
represents an untrimmed smooth affine-equivariant L-estimator of $\boldgreek\theta;$ for $\lambda\rightarrow 0$ we get the median affine-equivariant estimator, while  $\lambda\rightarrow 1$ gives a version of $L_{n2}$-type estimator. 
If $\lambda$ is chosen close to 1/2 and $k_n=o(\sqrt{n}),$ then tail $\sum_{j>k_n}w_{nj}^{(0)}$ converges exponentially to 0, implying a fast negligibility of the tail. {Parallelly, the weights $w_n(i)$ can be chosen as the nonincreasing rank scores $a_n(1)\geq a_n(2)\geq\ldots\geq a_n(n).$} 

To diminish the influence of the initial estimators, we can recursively continue in the same way: Put $\mathbf L_n^{(1)}=\mathbf L_n$ and define in the next step:
\begin{eqnarray*}
&&\mathbf A_n^{(1)}=\sum_{i=1}^n(\mathbf X_i-\mathbf L_n^{(1)})(\mathbf X_i-\mathbf L_n^{(1)})^{\top}\\
&&d_{ni}^{(1)}=(\mathbf X_i- \mathbf L_n^{(1)})^{\top}(\mathbf A_n^{(1)})^{-1}(\mathbf X_i- \mathbf L_n^{(1)})\\
&&R_{ni}^{(1)}=\sum_{j=1}^nI[d_{nj}^{(1)}\leq d_{ni}^{(1)}], \;  i=1,\ldots,n, \quad \mathbf R_n^{(1)}=(R_{n1}^{(1)},\ldots,R_{nn}^{(1)})^{\top}.
\end{eqnarray*} 
The second-step estimator is $\mathbf L_n^{(2)}=\sum_{i=1}^nw_n(R_{ni}^{(1)})\mathbf X_i.$ In this way we proceed, so at the $r$-th step we define 
$\mathbf A_n^{(r)}, \; d_{ni}^{(r)}, \; 1\leq i\leq n$ and the ranks $\mathbf R_n^{(r)}$ analogously, and get the $r$-step estimator
\begin{equation}\label{8.5.13}
\mathbf L_n^{(r)}=\sum_{i=1}^nw_n(R_{ni}^{(r-1)})\mathbf X_i, \; r\geq 1.
\end{equation}
Note that the $d_{ni}^{(r)}$ are affine-invariant for every $1\leq i\leq n$ and for every $r\geq 0.$ Hence,
applying an affine transformation $\mathbf Y_i=\mathbf B\mathbf X_i+\mathbf b, \; \mathbf b\in\mathbb R_p, \; \mathbf B$ positive definite, we see that
\begin{equation}\label{8.5.14}
\mathbf L_n^{(r)}(\mathbf Y_1,\ldots,\mathbf Y_n)=\mathbf B \mathbf L_n^{(r)}(\mathbf X_1,\ldots,\mathbf X_n)+\mathbf b.
\end{equation}
Hence, the estimating procedure preserves the affine equivariance at each step and $\mathbf L_n^{(r)}$ is an affine-equivariant L-estimator of $\boldgreek\theta$ for every $r.$  \\

The algorithm proceeds as follows:
\begin{description}
	\item{(1)}~ Calculate $\overline{\mathbf X}_n$ and $\mathbf A_n^{(0)}=\sum_{i=1}^n(\mathbf X_i-\overline{\mathbf X}_n)(\mathbf X_i-\overline{\mathbf X}_n)^{\top}.$
	\item{(2)}~ Calculate $d_{ni}^{(0)}=(\mathbf X_i-\overline{\mathbf X}_n)^{\top}(\mathbf A_n^{(0)})^{-1}(\mathbf X_i-\overline{\mathbf X}_n), \; 1\leq i\leq n.$ 
	\item{(3)}~ Determine the rank $R_{ni}^{(0)}$ of $d_{ni}^{(0)}$ among $d_{n1}^{(0)},\ldots,d_{nn}^{(0)}, \; i=1,\ldots,n.$
	\item{(4)}~ Calculate the scores $a_n(i), \; i=1,\ldots,n$ 
	\item{(5)}~ Calculate the first-step estimator $\mathbf L_n^{(1)}=\sum_{i=1}^na_n(R_{ni}^{(0)})\mathbf X_i.$
	\item{(6)}~ $\mathbf A_n^{(1)}=\sum_{i=1}^n(\mathbf X_i-\mathbf L_n^{(1)})(\mathbf X_i-\mathbf L_n^{(1)})^{\top}.$
	\item{(7)}~ $d_{ni}^{(1)}=(\mathbf X_i-\mathbf L_n^{(1)})^{\top}(\mathbf A_n^{(1)})^{-1}(\mathbf X_i-\mathbf L_n^{(1)}), \; 1\leq i\leq n.$
	\item{(8)}~ $R_{ni}^{(1)} =$ the rank of $d_{ni}^{(1)}$ among $d_{n1}^{(1)},\ldots,d_{nn}^{(1)}, \; i=1,\ldots,n.$
	\item{(9)}~ $\mathbf L_n^{(2)}=\sum_{i=1}^na_n(R_{ni}^{(1)})\mathbf X_i.$
	\item{(10)}~Repeat the steps (6)--(9).
\end{description}

\noindent

The estimator $\mathbf L_{n}^{(r)}$ is a linear combination of order statistics corresponding to independent random vectors $\mathbf X_1,\ldots,\mathbf X_n,$ with random coefficients based on the exchangeable $d_{ni}^{(r)}.$ The asymptotic distribution of $\mathbf L_{n}^{(r)}$  under fixed $r$ and for $n\rightarrow\infty$ is a problem for a future study, along with the asymptotic distribution of the $d_{ni}^{(r)}$ and of the rank statistics. 
For the moment, let us briefly recapitulate some asymptotic properties of the $d_{ni}^{(r)}.$  Note that $\sum_{i=1}^nd_{ni}^{(r)}=p \quad \forall~r\geq 0,$ and that the $d_{ni}^{(r)}$ are exchangeable nonnegative random variables {with a constant sum} and $\mathbb E(d_{ni}^{(r)})=\frac pn$ for every $r\geq 0.$ Let $\delta_{ni}^{(r)}=(\mathbf X_i-\mathbf L_{n}^{(r)})^{\top}\boldgreek\Sigma^{-1}(\mathbf X_i-\mathbf L_{n}^{(r)})$ and $\delta_i^*=(\mathbf X_i-\boldgreek\theta)^{\top}\boldgreek\Sigma^{-1}(\mathbf X_i-\boldgreek\theta), \; 1\leq i\leq n, \; r\geq 1$ {where $\boldgreek{\Sigma}$ is the covariance matrix of $\mathbf X_1.$} Let $G_n^{(r)}(y)=P\{nd_{ni}^{(r)}\leq y\}$ be the distribution function of the ${n}d_{ni}^{(r)}$ {and let $\widehat{G}_n^{(r)}(y)=n^{-1}\sum_{i=1}^nI[nd_{ni}^{(r)}\leq y], \; y\in\mathbb R^+$ be the empirical distribution function. Side by side, let $G_{nr}^*(y)=P\{\delta_{ni}^{(r)}\leq y\}$ and $G^*(y)=P\{\delta_i^*\leq y\}$ be the distribution function of  $\delta_{ni}^{(r)}$
and $\delta_i^*$ respectively. By the Slutzky theorem, 
$$|G_{nr}^*(y)-G^*(y)|\rightarrow 0 \; \mbox{ as } \; \ny.$$
Moreover, by the Courant theorem,
$$\rm{Ch}_{min}(\mathbf A\mathbf B^{-1})=\inf_{\mathbf x}~\frac{\mathbf x^{\top}\mathbf A\mathbf x}{\mathbf x^{\top}\mathbf B\mathbf x}\leq \sup_{\mathbf x}~\frac{\mathbf x^{\top}\mathbf A\mathbf x}{\mathbf x^{\top}\mathbf B\mathbf x}=\rm{Ch}_{max}(\mathbf A\mathbf B^{-1}),$$
we have
$$\max_{1\leq i\leq n}\left|\frac{nd_{ni}^{(r)}}{\delta_{ni}^{(r)}}-1\right|\leq \max\left\{\left|\rm{Ch}_{max}\left(\frac 1n(\mathbf A_n^{(r)})^{-1}\boldgreek\Sigma\right)-1\right|,\left|\rm{Ch}_{min}\left(\frac 1n(\mathbf A_n^{(r)})^{-1}\boldgreek\Sigma\right)-1\right|\right\}$$ 
so that $\left[\frac 1n\mathbf A_n^{(r)}\stackrel{p}{\rightarrow}\boldgreek\Sigma\right] \; \Longrightarrow |\widehat{G}_n^{(r)}-G_{nr}^*|\rightarrow 0.$ In a similar way, $|\delta_{ni}^{(r)}-\delta_i^*|\ll \|\mathbf L_{nk}^{(r)}-\boldgreek\theta\|,$ where the right-hand side is $O_p(n^{-1/2}).$
Because $d_{ni}^{(r)}$ are exchangeable, bounded and nonnegative random variables, 
one can use the Hoeffding (1963) inequality  
to verify that for every $c_n>0,$ there exist positive constants $K_1$ and $\nu$ for which 
$$P\left\{|\widehat{G}_n^{(r)}(y)-G_n^{(r)}(y)|>c_n\right\}\leq K_1\rm{e}^{-\nu n c_n^2}.$$
Thus, using $c_n=O\left(n^{1/2}\log n\right)$ can make the right-hand side to converge at any power rate with $\ny.$ This leads to the following lemma. 
\begin{lemma}\label{Lemma8.1}
As $\ny,$ 
\begin{eqnarray} \label{8.5.15}
&\sup_{d\in\mathbb R^+}\left\{|\widehat{G}_n^{(r)}(y)-G_n^{(r)}(y)-\widehat{G}_n^{(r)}(y^{\prime})+G_n^{(r)}(y^\prime)|: \; |y-y^{\prime}|\leq 
n^{-1/2}\sqrt{2\log n}\right\}&\nonumber\\
&\stackrel{a.s.}{=}O(n^{-\frac 34}\log n). &
\end{eqnarray} 
\end{lemma}
\begin{proof} [outline] The lemma follows from the Borel-Cantelli lemma, when we notice that
both $\widehat{G}_n^{(r)}(y)$ and $G_n^{(r)}(y)$ are $\nearrow$ in $y\in\mathbb R^+,$ and that $\widehat{G}_n^{(r)}(0)=G_n^{(r)}(0), \; \widehat{G}_n^{(r)}(\infty)=G_n^{(r)}(\infty)=1.$  
\end{proof}
\begin{theorem}\label{Theorem8.2} Let $$W_{nr}(t)=n^{-1/2}[\widehat{G}_n^{(r)}(G_n^{(r)-1}(t))-t], \; t\in[0,1]; \; W_{nr}=\{W_{nr}(t); \; 0\leq t\leq 1\}.$$
Then $W_{nr}\Rightarrow W$  in the Skorokhod $\mathcal D[0,1]$ topology, where $W$ is a Brownian Bridge on [0,1]. 
\end{theorem} 
}
\begin{proof} [outline] The tightness part of the proof follows from Lemma \ref{Lemma8.1}. For the convergence of finite-dimensional distributions, we appeal to the central limit theorem for interchangeable random variables of Chernoff and Teicher (1958).
\end{proof}

{If the $\mathbf X_i$ have multinormal distribution, then $\delta_i^*$ has the Chi-squared distribution with $p$ degrees of freedom. If  $\mathbf X_i$ is elliptically symmetric, then its density depends on $h(\|\mathbf x-\boldgreek\theta\|_{\Sigma}),$ with $h(y), \; y>0$ depending only on the norm $\|y\|.$ If $p\geq 2,$ as it is in our case, it may be reasonable to assume that $H(y)=\int_0^yh(u)du$ behaves as $\sim y^{p/2}$ (or higher power) for $y\rightarrow 0.$ Thus $y\sim[H(y)]^{2/p}$ (or $[H(y)]^r, \; r\leq 2/p$) for $y\rightarrow 0.$ On the other hand, since our choice is $k_n=o(n),$ the proposed estimators $\mathbf L_{n1}$ and $\mathbf L_{n2}$ both depend on the $\mathbf X_i$ with $d_{ni}$ of lower rank ($R_{ni}\leq k_n$ or $n^{-1}R_{ni}\leq n^{-1}k_n\rightarrow 0$). Hence, both $\mathbf L_{n1}$ and $\mathbf L_{n2}$ are close to the induced vector $\mathbf X_{[1]}$ where $[1]=\{i: \; R_{ni}=1\}.$
If the initial estimator is chosen as $\mathbf X_{[1]}$ and $\mathbf A_{n[1]}=n^{-1}\sum_{i=1}^n(\mathbf X_i-\overline{\mathbf X}_{[1]})(\mathbf X_i-\overline{\mathbf X}_{[1]})^{\top},$ then the iteration process will be comparatively faster than if we start with the initial estimators $\overline{\mathbf X}_n$ and $n^{-1}\mathbf A_n^{(0)}$.

The proposed $\mathbf L_{n1}, \; \mathbf L_{n2}$ are both affine equivariant and robust. {If we define  the D-efficiency 
\begin{equation}\label{defficiency}
\mathbf D_n^{(r)}=\left(\frac{|\mathbf A_n^{(r)}|}{|\mathbf A_n^{(0)}|}\right)^{1/p}, \quad r \leq 1,
\end{equation}
then it} will be slightly better than that of the spatial median; the classical $\overline{\mathbf X}_n$ has the best efficiency for multinormal distribution but it is 
much less robust than $\mathbf L_{n1}$ and $\mathbf L_{n2}.$
}

\section{Numerical illustration}
\label{sec:3}
\subsection{Multivariate normal distribution}
The procedure is illustrated on  samples of size $ n=100$ simulated from the normal distribution $\mathcal N_3(\boldgreek\theta,\boldgreek\Sigma)$ with  
 with
\begin{equation}\label{parameter}
\boldgreek\theta=\left(\begin{array}{r}
\theta_1\\
\theta_2\\
\theta_3\\
\end{array}\right)=\left(\begin{array}{r}
1\\
2\\
-1\\
\end{array}\right)\qquad \qquad \boldgreek\Sigma=\left[\begin{array}{ccc}
                                     1 &~1/2&~1/2\\
                                     1/2&~ 1 &~1/2\\
                                     1/2&~ 1/2 &~1\\
                                     \end{array}\right]
\end{equation}                                     
and each time the affine-equivariant trimmed $\mathbf L_{n1}$-estimator ($k_n=15$) and affine-equivariant $\mathbf L_{n2}$-estimator were calculated in 10 iterations of the initial estimator.  
5\,000 replications of the model were simulated and also the mean
 was computed, for the sake of comparison. Results  are summarized in Table 1. Figure 1 illustrates the distribution of
estimated  parameters $\theta_1$, $\theta_2$, $\theta_3$  for various iterations of $\mathbf L_{n1}$-estimator and $\mathbf L_{n2}$-estimator and compares them with the mean and median.
Tables 2-3 and Figure 2 compare the D-efficiency of proposed estimators. 

The Mahalanobis distance  is also illustrated. One sample of size $ n=100$ was simulated   from  the bivariate normal distribution with the above parameters. Afterwards, the Mahalanobis distances $d_{ii} = (X_i - \bar{X} )^T S_n^{-1} (X_i - \bar{X}),\  i = 1,\ldots, n$ were calculated. They represent $n$  co-axial ellipses centered at $\bar{X}$ - see Figure 3 (black ellipses).  The modified Mahalanobis distances  replaced $\bar{X}$ by  the affine-equivariant trimmed $\mathbf L_{n1}$-estimator  with $k_n=15$ (see the blue ellipses on Figure 3) and  affine-equivariant $\mathbf L_{n2}$-estimator (see the red ellipses on Figure 3) with analogous modification of $S_n$.


\bigskip




\begin{center}

\textit{Table 1. Normal distribution: 
The mean in the sample of 5 000 replications of  estimators $\mathbf L_{n1}$ (trimmed) and $\mathbf L_{n2},$ 
sample sizes  $n=100$}\\[3mm]

\begin{tabular}{|c|lll|lll|}
\hline
&&&&&&\\[-2mm]
i     & \multicolumn{3}{|c|}{${\mathbf L}_{n1}^{(i)}$}&\multicolumn{3}{|c|}{$\mathbf L_{n2}^{(i)}$}\\[1mm]
\hline
1 & 0.999607 & 2.001435 & -0.998297 & 0.999401 & 1.999796 & -0.999943 \\
  2 & 0.999473 & 2.001423 & -0.996185 & 0.999083 & 1.999584 & -0.999782 \\
  3 & 0.999519 & 2.000290 & -0.993274 & 0.998926 & 1.999509 & -0.999801 \\
  4 & 0.999435 & 2.000190 & -0.991901 & 0.998854 & 1.999496 & -0.999871 \\
  5 & 0.999474 & 2.000771 & -0.991295 & 0.998811 & 1.999476 & -0.999973 \\
  6 & 0.999646 & 2.001285 & -0.990964 & 0.998781 & 1.999471 & -1.000032 \\
  7 & 0.999926 & 2.001519 & -0.990829 & 0.998773 & 1.999470 & -1.000049 \\
  8 & 0.999952 & 2.001529 & -0.990803 & 0.998772 & 1.999472 & -1.000068 \\
  9 & 0.999939 & 2.001497 & -0.990745 & 0.998775 & 1.999489 & -1.000071 \\
  10 & 0.999853 & 2.001424 & -0.990711 & 0.998779 & 1.999497 & -1.000061 \\
   \hline
\end{tabular}

\end{center}




\begin{figure}[h]
\begin{center}
\includegraphics[height=7.3cm]{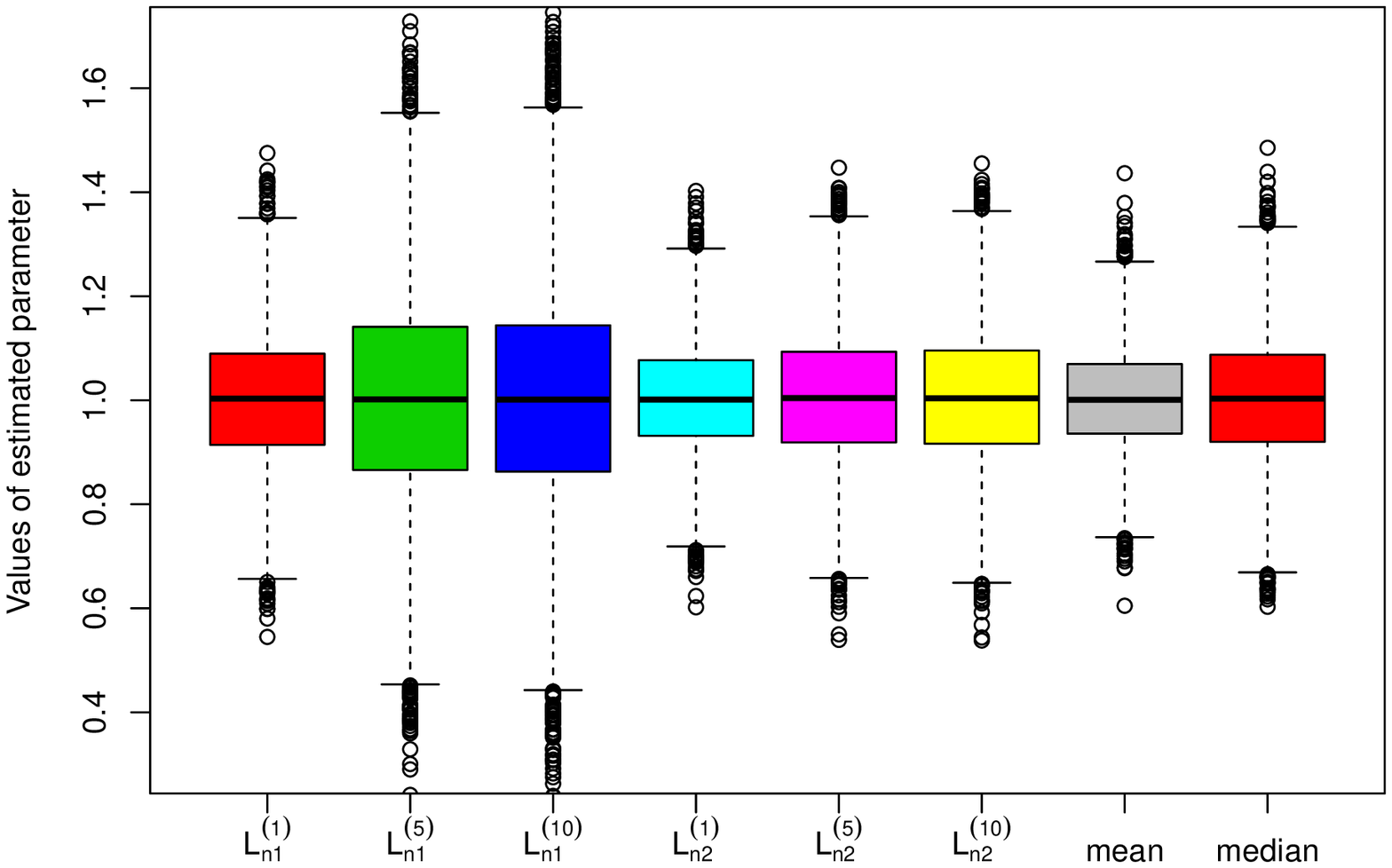}\\[-13mm]
\includegraphics[height=7.3cm]{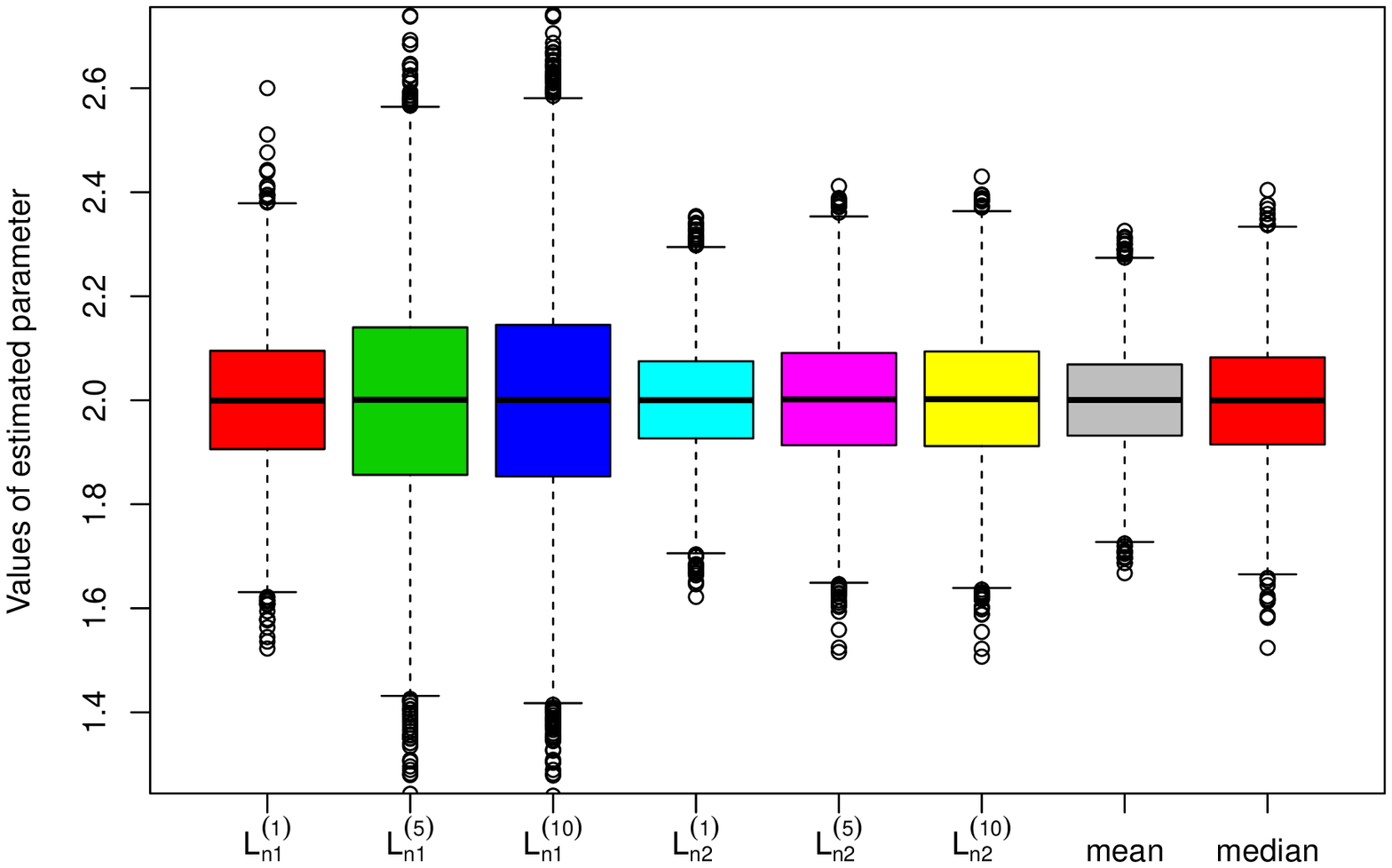}\\[-13mm]
\includegraphics[height=7.3cm]{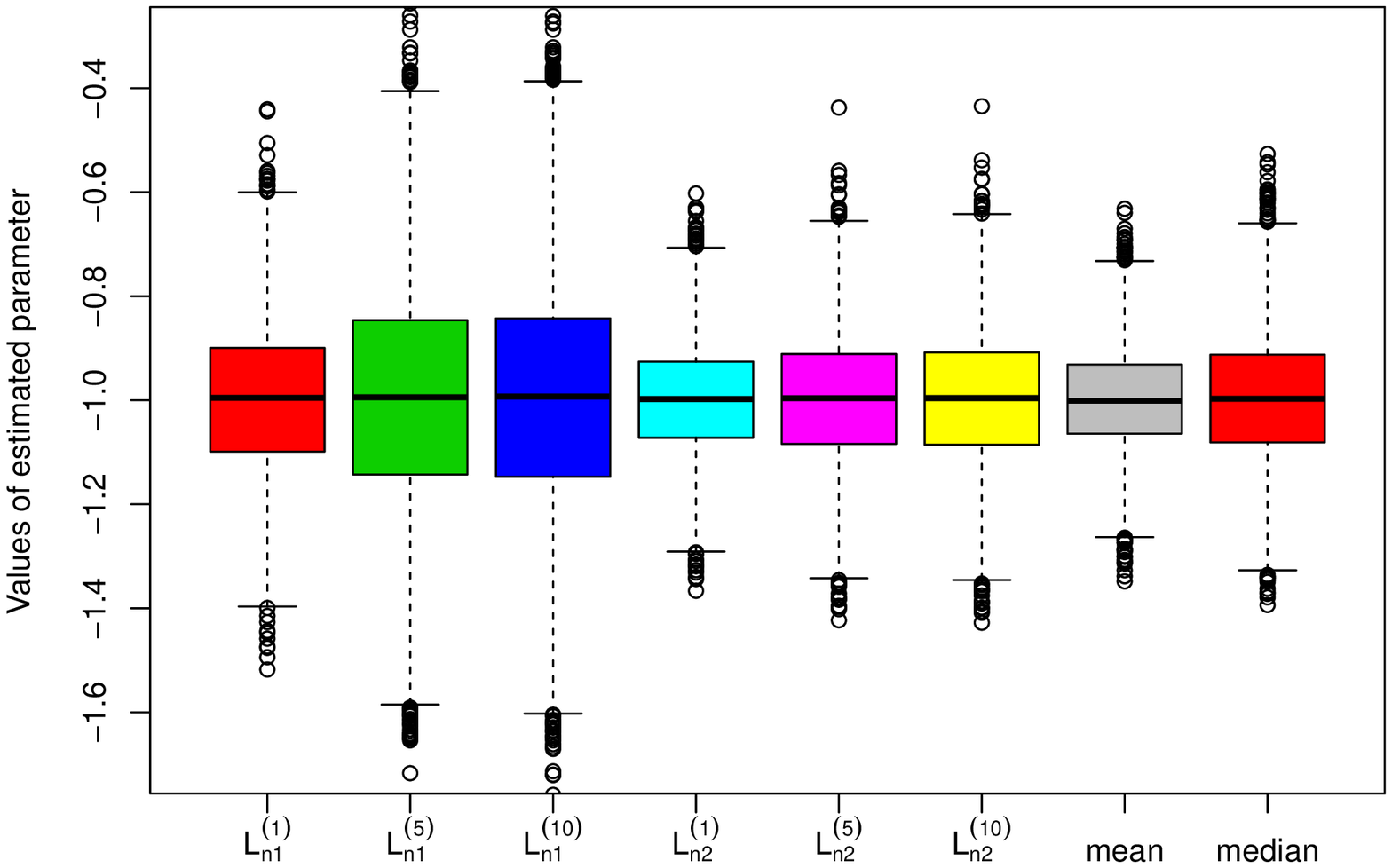}\\[-3mm]
\caption{Normal distribution: Box-plots of the 5\,000 estimated values of $\theta_1(=1)$ (top), $\theta_2(=2)$ (middle) and $\theta_3(=-1)$ (bottom)   for 
the $\mathbf L_{n1}^{(1)}$, $\mathbf L_{n1}^{(5)}$, $\mathbf L_{n1}^{(10)}$, $\mathbf L_{n2}^{(1)}$, $\mathbf L_{n2}^{(5)}$, $\mathbf L_{n2}^{(10)}$, mean and median.
}%
\label{figure1}
\end{center}
\end{figure}


\begin{center}
\textit{Table 2. Normal distribution:
The mean, median and minimum of D-efficiency in the sample of 5 000 replications of  estimators $\mathbf L_{n1}$ (trimmed) and $\mathbf L_{n2},$ 
sample sizes  $n=100$ }\\[3mm]
\begin{tabular}{|c|lll|lll|}
\hline
     &&&&&&\\[-2mm]
     & \multicolumn{3}{|c|}{${\mathbf L}_{n1}^{(i)}$ }&\multicolumn{3}{|c|}{$\mathbf L_{n2}^{(i)}$}\\[1mm] 
\hline
iterration & mean & median & minimum & mean & median & minimum\\[1mm] 
\hline   
2 & 1.000637 & 0.999615 & 0.871029 & 0.999862 & 0.999670 & 0.951057 \\ 
 3 & 1.001272 & 0.998593 & 0.818591 & 0.999625 & 0.999451 & 0.946809 \\ 
  4 & 1.001106 & 0.997822 & 0.793030 & 0.999392 & 0.999231 & 0.946575 \\ 
  5 & 1.000697 & 0.997283 & 0.794157 & 0.999205 & 0.999041 & 0.946332 \\ 
  6 & 1.000394 & 0.997211 & 0.793903 & 0.999078 & 0.998886 & 0.946022 \\ 
  7 & 1.000131 & 0.997122 & 0.793903 & 0.998997 & 0.998802 & 0.945613 \\ 
  8 & 0.999916 & 0.996964 & 0.793903 & 0.998951 & 0.998807 & 0.945032 \\ 
  9 & 0.999882 & 0.996924 & 0.793903 & 0.998921 & 0.998768 & 0.944518 \\ 
  10 & 0.999860 & 0.996924 & 0.793903 & 0.998899 & 0.998706 & 0.944272 \\ 
   \hline
\end{tabular}

\bigskip

\bigskip


\textit{Table 3. Normal distribution: 
The 25\%-quantile, 75\%-quantile and max of D-efficiency in the sample of 5000 replications of  estimators $\mathbf L_{n1}$ (trimmed) and $\mathbf L_{n2},$ 
sample sizes  $n=100$ }\\[3mm]
\begin{tabular}{|c|lll|lll|}
\hline
     &&&&&&\\[-2mm]
     & \multicolumn{3}{|c|}{${\mathbf L}_{n1}^{(i)}$ }&\multicolumn{3}{|c|}{$\mathbf L_{n2}^{(i)}$}\\[1mm] 
\hline   
\hline
iteration & 25\%-quantile & 75\%-quantile & max & 25\%-quantile & 75\%-quantile & max\\[1mm] 
\hline   
2 & 0.971669 & 1.028444 & 1.169137 & 0.990914 & 1.008819 & 1.057115 \\ 
 3 & 0.960891 & 1.039437 & 1.295245 & 0.989675 & 1.009489 & 1.065753 \\ 
  4 & 0.956068 & 1.042774 & 1.294952 & 0.989358 & 1.009426 & 1.068542 \\ 
  5 & 0.953186 & 1.043542 & 1.301147 & 0.989098 & 1.009239 & 1.069236 \\ 
  6 & 0.952193 & 1.044435 & 1.305327 & 0.989000 & 1.009182 & 1.069437 \\ 
  7 & 0.951535 & 1.044942 & 1.326996 & 0.988916 & 1.009142 & 1.069600 \\ 
  8 & 0.951441 & 1.044394 & 1.330791 & 0.988839 & 1.009109 & 1.069226 \\ 
  9 & 0.951452 & 1.044562 & 1.330791 & 0.988802 & 1.009087 & 1.069236 \\ 
  10 & 0.951356 & 1.044749 & 1.330791 & 0.988771 & 1.009062 & 1.069278 \\ 
 \hline
\end{tabular}
\end{center}


\newpage

\begin{figure}
\begin{center}
\includegraphics[height=7.3cm]{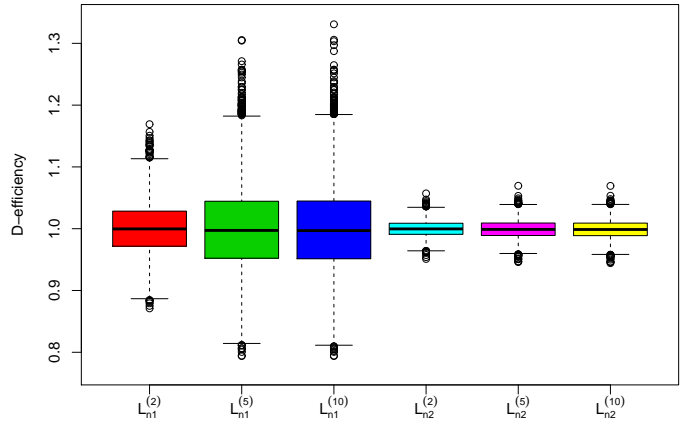}\\[-3mm]
\caption{Normal distribution: Box-plots of the 5\,000 estimated values of D-efficiency   for 
the $\mathbf L_{n1}^{(2)}$, $\mathbf L_{n1}^{(5)}$, $\mathbf L_{n1}^{(10)}$, $\mathbf L_{n2}^{(2)}$, $\mathbf L_{n2}^{(5)}$, $\mathbf L_{n2}^{(10)}$.
}%
\label{figure2}
\end{center}
\end{figure}

\begin{figure}
\begin{center}
\begin{tabular}{cc}
\includegraphics[height=5.5cm]{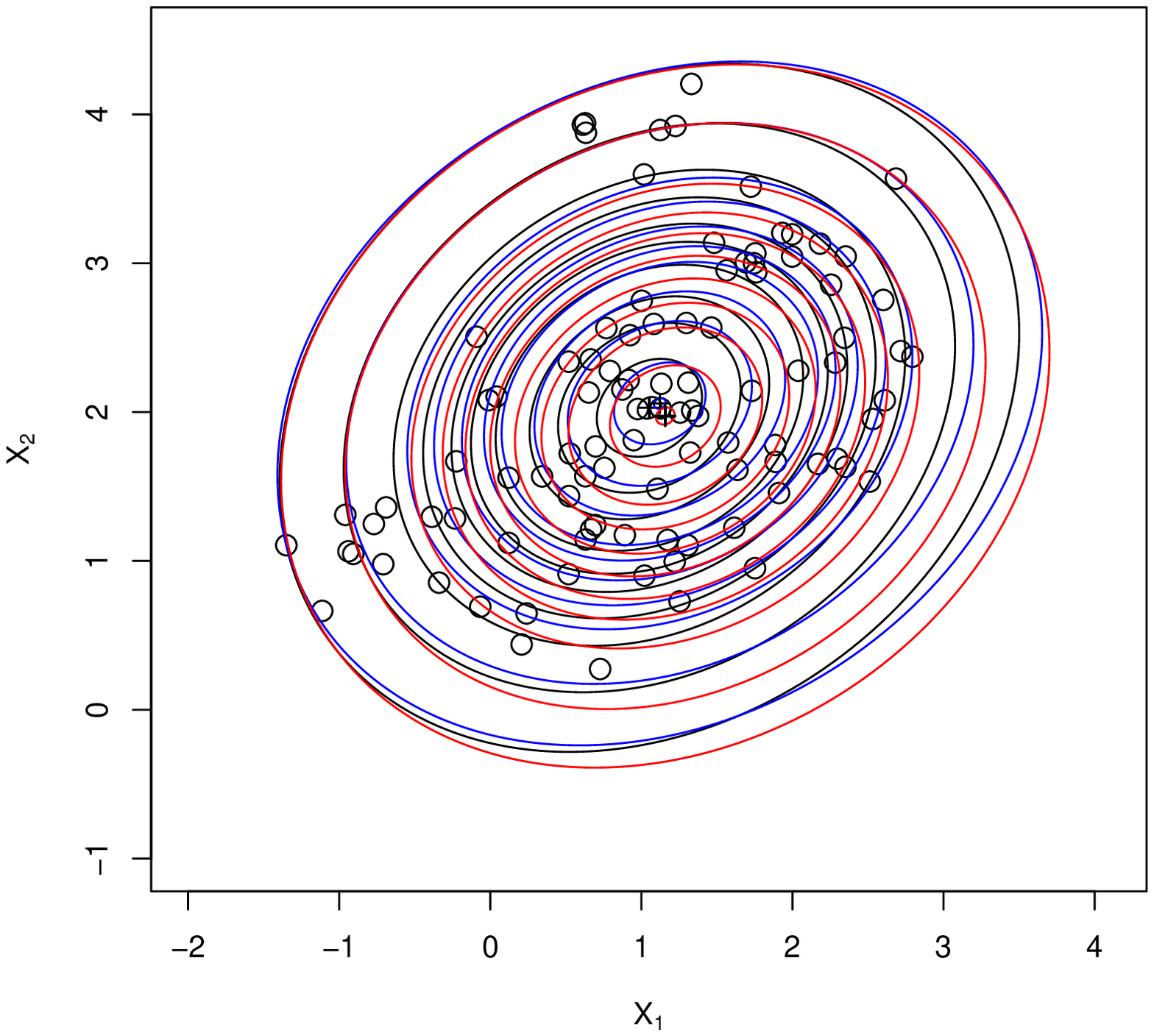} & 
\includegraphics[height=5.5cm]{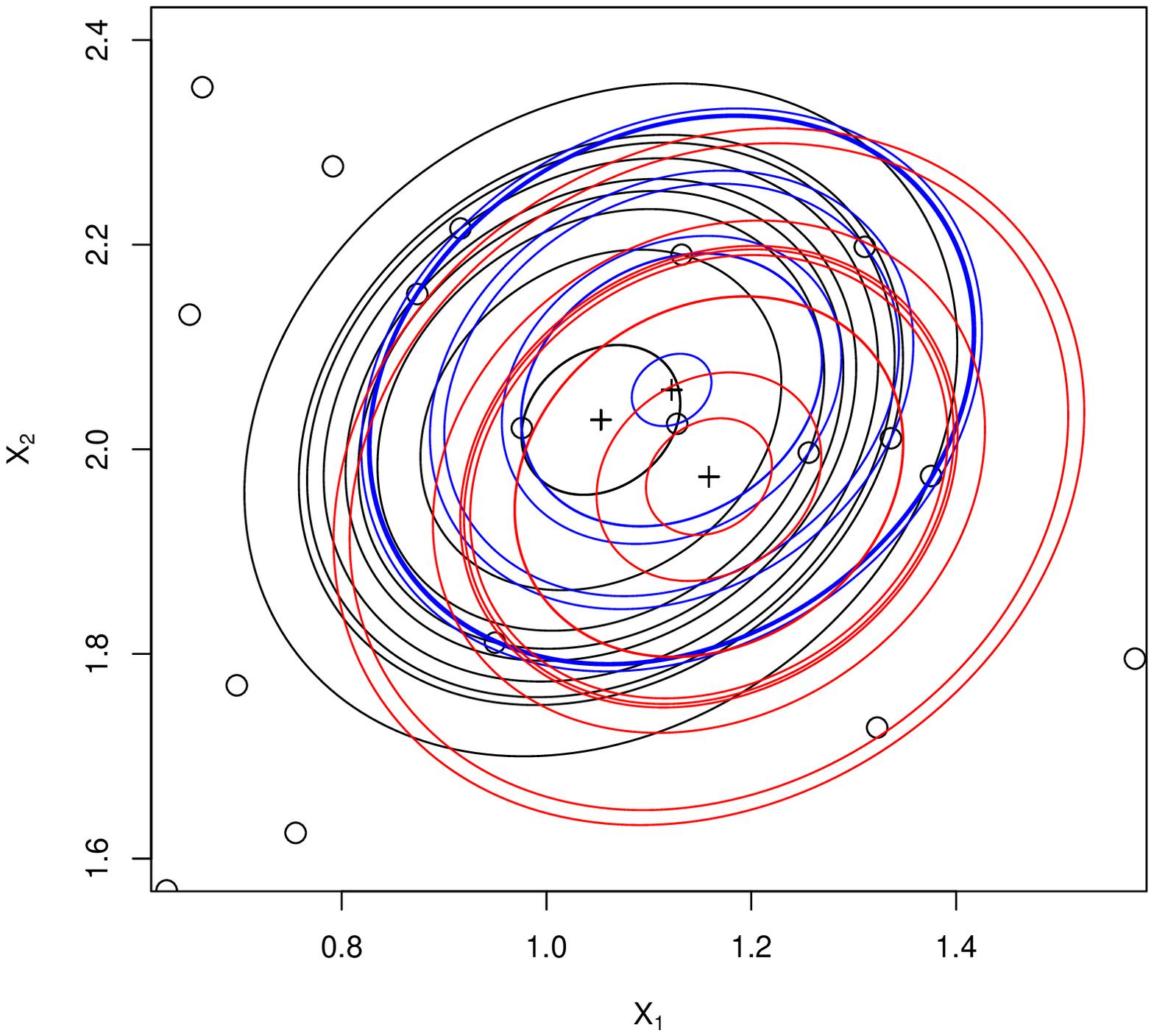}
\end{tabular}
\caption{Normal distribution: Mahalanobis distances represented by co-axial ellipses centered at the mean $\bar{X}$ (black), at the trimmed $\mathbf L_{n1}$-estimator (blue) and at the $\mathbf L_{n2}$-estimator (red).   All simulated bivariate data  with the every tenth contour are illustrated on the left,  detail of the center with the first ten contours is on the right.
}%
\label{figure11}

\end{center}
\end{figure}

\newpage
\subsection{Multivariate t-distribution}
Similarly, we illustrate the procedure on samples of size $ n=100$ simulated from the multivariate $t$ distribution with 3 degree of freedom $t_3(\boldgreek\theta,\boldgreek\Sigma),$ with the same parameters as in (\ref{parameter}). 
Each time, 10 iterations of affine-equivariant trimmed $\mathbf L_{n1}$-estimator ($k_n=15$) and of affine-equivariant $L_{n2}$-estimator, started from the initial estimator, were calculated.  
5\,000 replications of the model were simulated and the mean
 was computed, for the sake of comparison. Results  are summarized in Table 4. Figure 4 illustrates the distribution of
estimated  parameters $\theta_1$, $\theta_2,$ $\theta_3$  for various iterations of $\mathbf L_{n1}$-estimator and $\mathbf L_{n2}$-estimator and compares them with the mean and median.
The Tables 5-6 and Figure 5 compare the D-efficiencies of the proposed estimators and Figure 6  illustrates the  Mahalanobis distances.

\bigskip

\begin{center}
\textit{Table 4. $t$-distribution: 
The mean in the sample of 5 000 replications of  estimators $\mathbf L_{n1}$ (trimmed) and $\mathbf L_{n2},$ 
sample sizes  $n=100$}\\[3mm]

\begin{tabular}{|c|lll|lll|}
\hline
&&&&&&\\[-2mm]
i     & \multicolumn{3}{|c|}{${\mathbf L}_{n1}^{(i)}$}&\multicolumn{3}{|c|}{$\mathbf L_{n2}^{(i)}$}\\[1mm]
\hline
1 & 1.004760 & 2.002252 & -0.992618 & 1.003081 & 2.001199 & -0.998781 \\ 
  2 & 1.005034 & 2.001851 & -0.991150 & 1.002154 & 2.000044 & -0.999996 \\ 
  3 & 1.005473 & 2.001430 & -0.991330 & 1.001744 & 1.999661 & -1.000627 \\ 
  4 & 1.005334 & 2.000732 & -0.992535 & 1.001594 & 1.999526 & -1.000909 \\ 
  5 & 1.005351 & 2.000951 & -0.993532 & 1.001540 & 1.999475 & -1.001028 \\ 
  6 & 1.005007 & 2.001127 & -0.994361 & 1.001528 & 1.999452 & -1.001069 \\ 
  7 & 1.004636 & 2.001009 & -0.994817 & 1.001522 & 1.999447 & -1.001086 \\ 
  8 & 1.004520 & 2.000930 & -0.995011 & 1.001524 & 1.999443 & -1.001090 \\ 
  9 & 1.004466 & 2.000911 & -0.995172 & 1.001526 & 1.999442 & -1.001090 \\ 
  10 & 1.004445 & 2.000821 & -0.995221 & 1.001527 & 1.999444 & -1.001086 \\ 
   \hline
\end{tabular}

\bigskip

\bigskip






\textit{Table 5. $t$-distribution:
The mean, median and minimum of D-efficiency in the sample of 5 000 replications of  estimators $\mathbf L_{n1}$ (trimmed) and $\mathbf L_{n2},$ 
sample sizes  $n=100$ }\\[3mm]
\begin{tabular}{|c|lll|lll|}
\hline
     &&&&&&\\[-2mm]
     & \multicolumn{3}{|c|}{${\mathbf L}_{n1}^{(i)}$ }&\multicolumn{3}{|c|}{$\mathbf L_{n2}^{(i)}$}\\[1mm] 
\hline
iterration & mean & median & minimum & mean & median & minimum\\[1mm] 
\hline   
2 & 1.001813 & 1.000857 & 0.896048 & 1.000812 & 1.000303 & 0.857210 \\ 
3 & 1.001827 & 1.001157 & 0.824105 & 1.000556 & 1.000109 & 0.845643 \\ 
4 & 1.001377 & 1.000619 & 0.810082 & 1.000360 & 0.999912 & 0.844578 \\ 
5 & 1.000887 & 0.999840 & 0.796776 & 1.000260 & 0.999766 & 0.844182 \\ 
6 & 1.000372 & 0.999121 & 0.777458 & 1.000214 & 0.999723 & 0.843850 \\ 
7 & 1.000024 & 0.999086 & 0.756777 & 1.000196 & 0.999740 & 0.843933 \\ 

8 & 0.999881 & 0.998921 & 0.756555 & 1.000187 & 0.999714 & 0.843928 \\ 
9 & 0.999806 & 0.998907 & 0.756655 & 1.000184 & 0.999726 & 0.843917 \\ 
10& 0.999753 & 0.998921 & 0.756655 & 1.000183 & 0.999726 & 0.843897 \\ 

   \hline
\end{tabular}

\bigskip



\textit{Table 6. $t$-distribution: 
The 25\%-quantile, 75\%-quantile and max of D-efficiency in the sample of 5 000 replications of  estimators $\mathbf L_{n1}$ (trimmed) and $\mathbf L_{n2},$ 
sample sizes  $n=100$ }\\[3mm]
\begin{tabular}{|c|lll|lll|}
\hline
     &&&&&&\\[-2mm]
     & \multicolumn{3}{|c|}{${\mathbf L}_{n1}^{(i)}$ }&\multicolumn{3}{|c|}{$\mathbf L_{n2}^{(i)}$}\\[1mm] 
\hline   
\hline
iterration & 25\%-quantile & 75\%-quantile & max & 25\%-quantile & 75\%-quantile\\[1mm] 
\hline   
2 & 0.980004 & 1.023251 & 1.157852 & 0.986053 & 1.014785 & 1.194658 \\ 
3 & 0.972157 & 1.030707 & 1.212343 & 0.983882 & 1.016369 & 1.212103 \\ 
4 & 0.967982 & 1.032541 & 1.214165 & 0.983518 & 1.016634 & 1.214597 \\ 
5 & 0.967034 & 1.032940 & 1.222945 & 0.983380 & 1.016629 & 1.215499 \\ 
6 & 0.966116 & 1.032850 & 1.260563 & 0.983366 & 1.016565 & 1.215862 \\ 
7 & 0.965784 & 1.032895 & 1.261640 & 0.983391 & 1.016572 & 1.216013 \\ 
8 & 0.965561 & 1.032746 & 1.261675 & 0.983406 & 1.016537 & 1.216332 \\ 
9 & 0.965436 & 1.032643 & 1.261675 & 0.983380 & 1.016522 & 1.216431 \\ 
10& 0.965404 & 1.032680 & 1.261675 & 0.983406 & 1.016522 & 1.216579 \\

 \hline
\end{tabular}
\end{center}
Although $L_n^{(1)}$ resembles the NN-estimator, its behavior for $t$-distribution reveals its robustness no less than $L_n^{(2)}.$ For multivariate normal distribution, both $L_n^{(1)}$ and $L_n^{(2)}$ seem to be doing well against outliers.  Figures \ref{figure3} and \ref{figure4} illustrate this feature in a visible way.

\bigskip

\begin{figure}[h]
\begin{center}
\includegraphics[height=7.3cm]{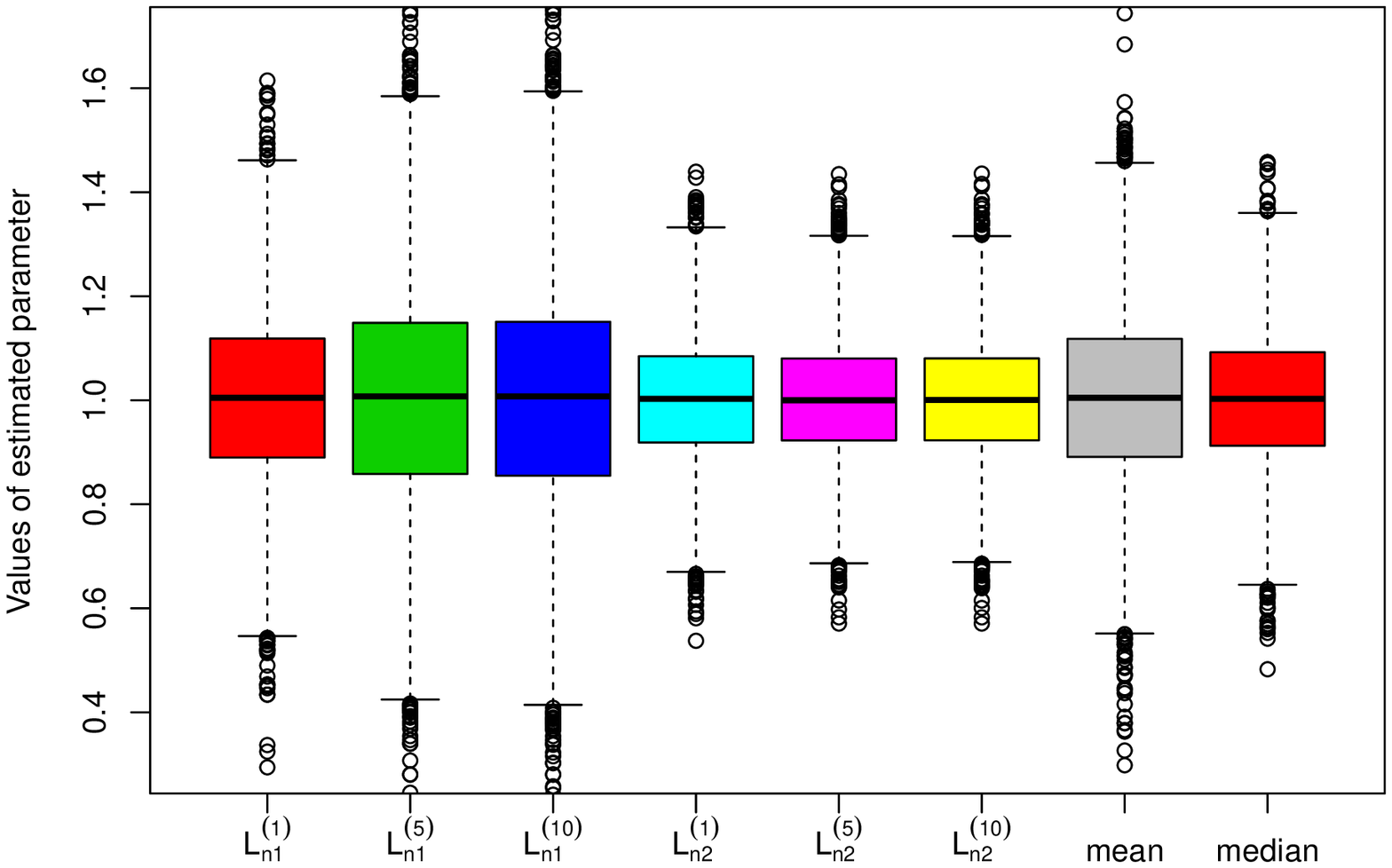}\\[-13mm]
\includegraphics[height=7.3cm]{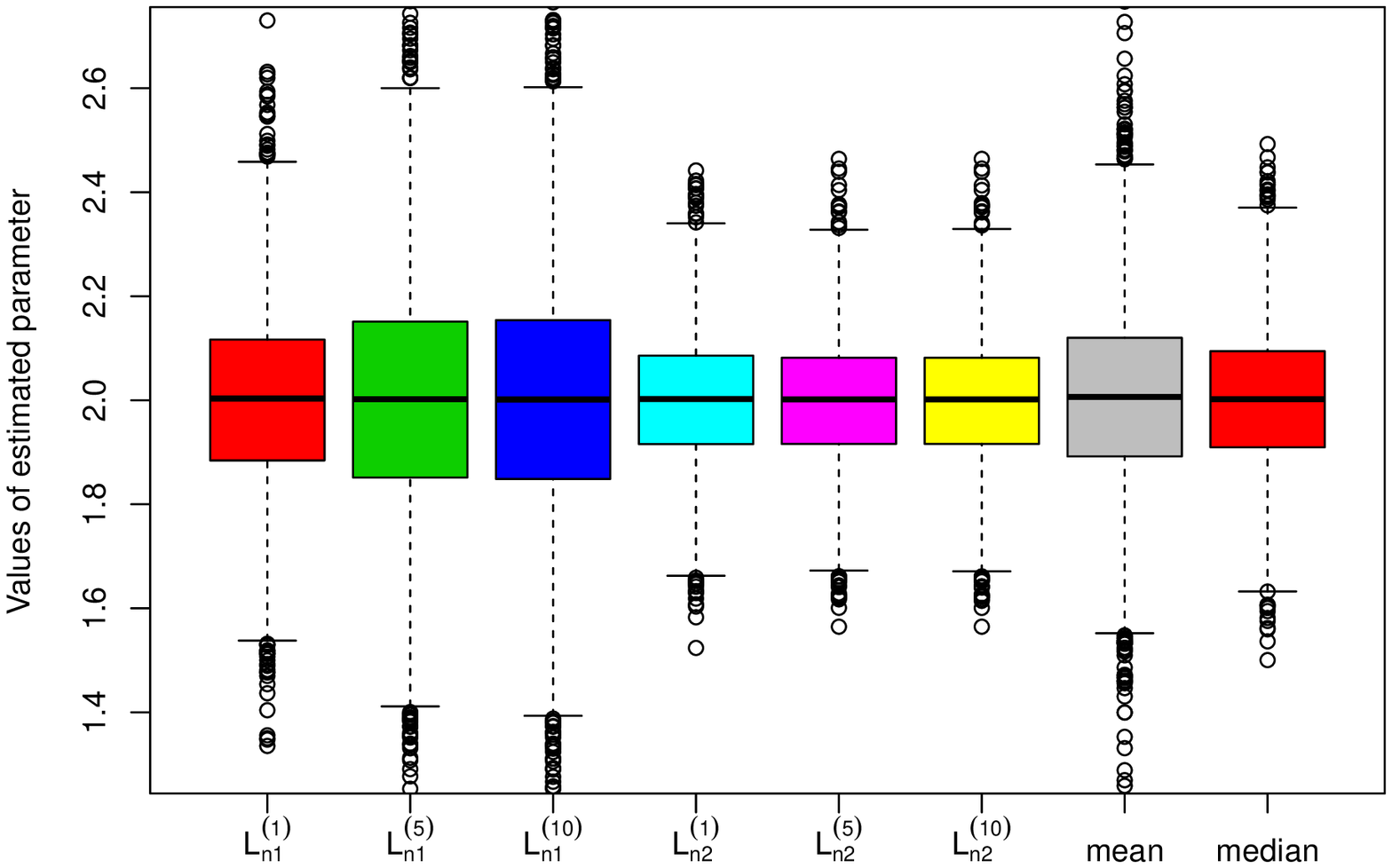}\\[-13mm]
\includegraphics[height=7.3cm]{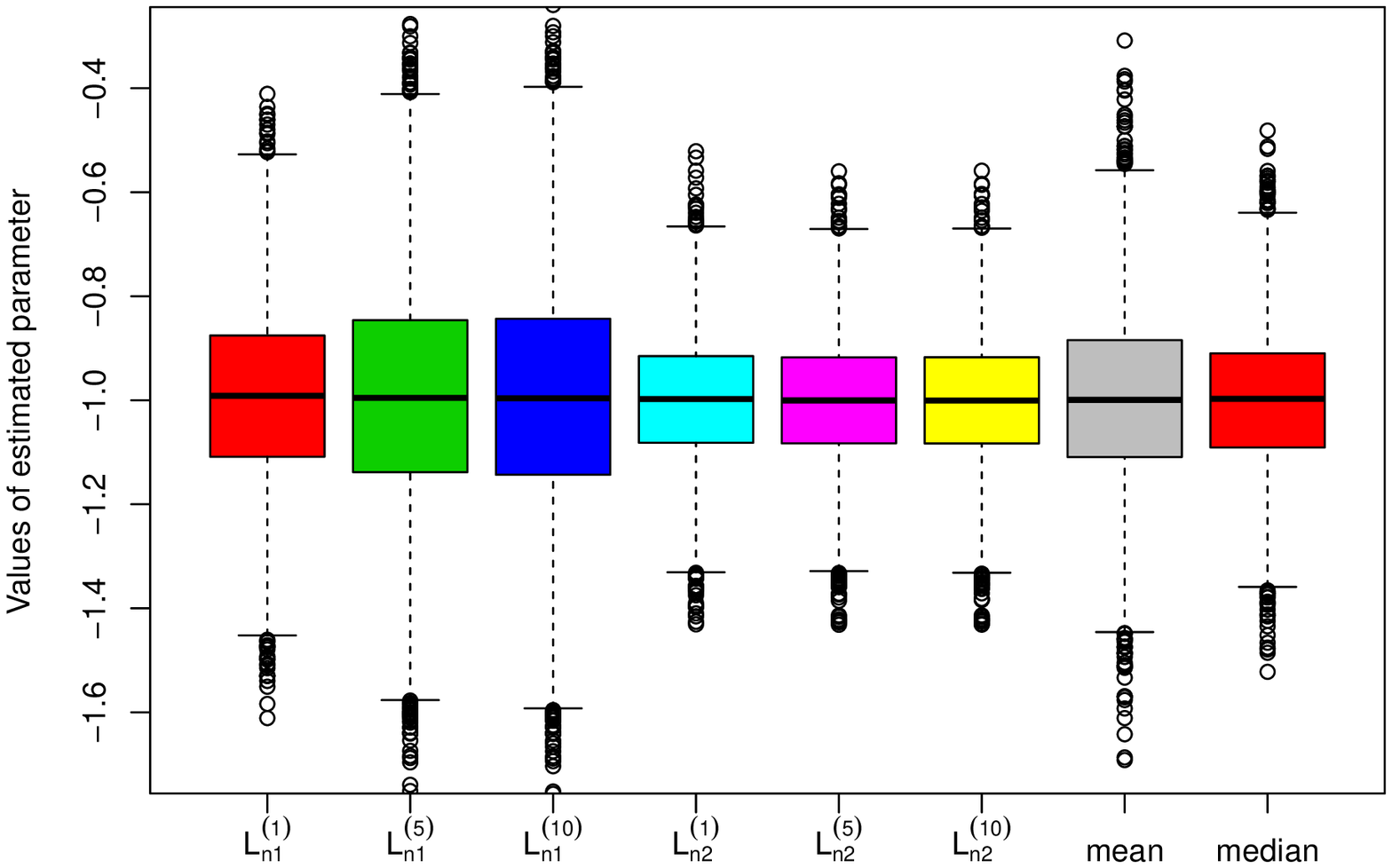}\\[-3mm]
\caption{$t$-distribution: Box-plots of the 5\,000 estimated values of $\theta_1(=1)$ (top), $\theta_2(=2)$ (middle) and $\theta_3(=-1)$ (bottom)   for 
the $\mathbf L_{n1}^{(1)}$, $\mathbf L_{n1}^{(5)}$, $\mathbf L_{n1}^{(10)}$, $\mathbf L_{n2}^{(1)}$, $\mathbf L_{n2}^{(5)}$, $\mathbf L_{n2}^{(10)}$, mean and median.
}%
\label{figure3}
\end{center}
\end{figure}

\bigskip

\begin{figure}
\begin{center}
\includegraphics[height=7.3cm]{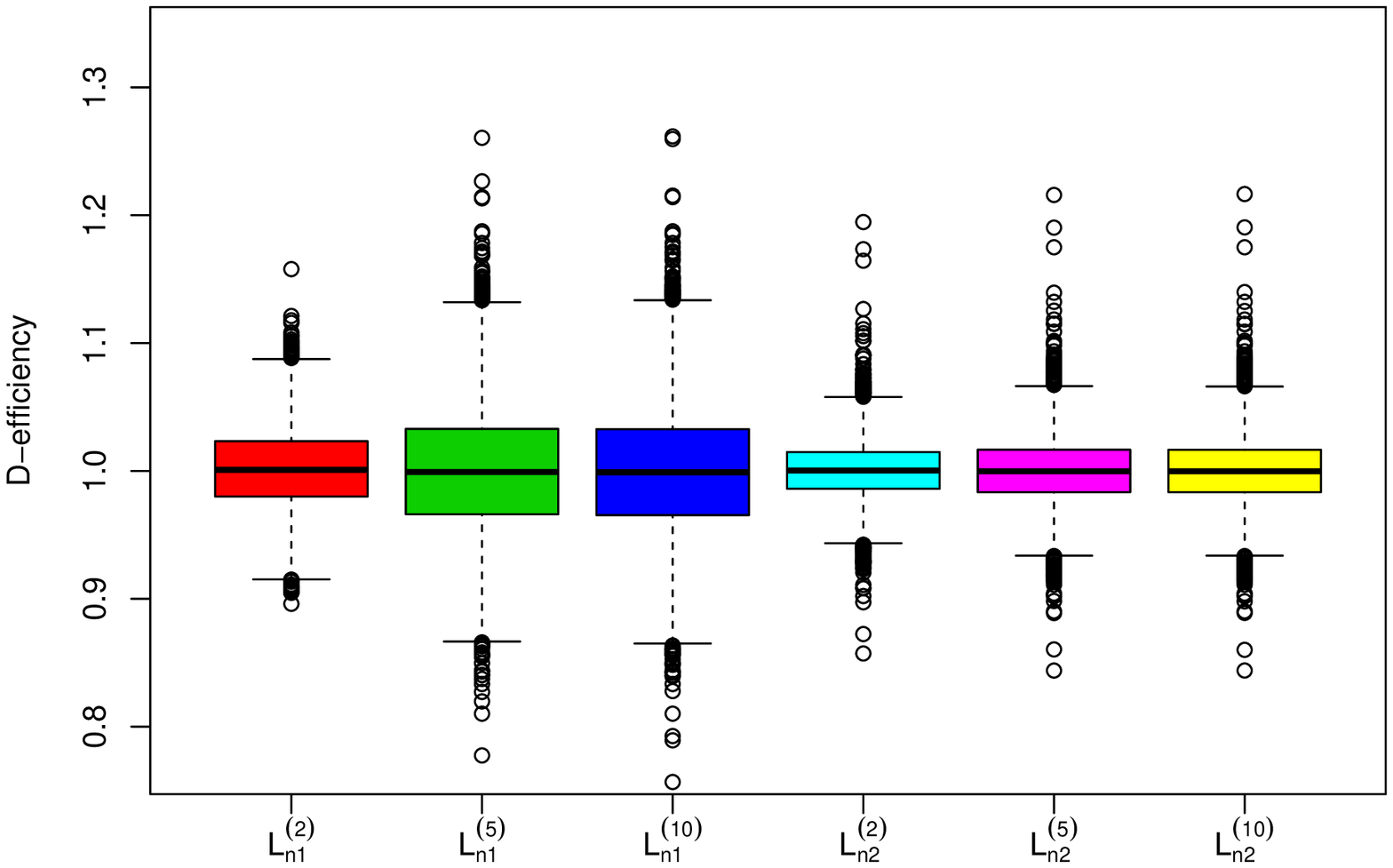}\\[-3mm]
\caption{$t$-distribution: Box-plots of the 5\,000 estimated values of D-efficiency   for 
the $\mathbf L_{n1}^{(2)}$, $\mathbf L_{n1}^{(5)}$, $\mathbf L_{n1}^{(10)}$, $\mathbf L_{n2}^{(2)}$, $\mathbf L_{n2}^{(5)}$, $\mathbf L_{n2}^{(10)}$.
}%
\label{figure4}
\end{center}
\end{figure}

\begin{figure}
\begin{center}
\begin{tabular}{cc}
\includegraphics[height=5.5cm]{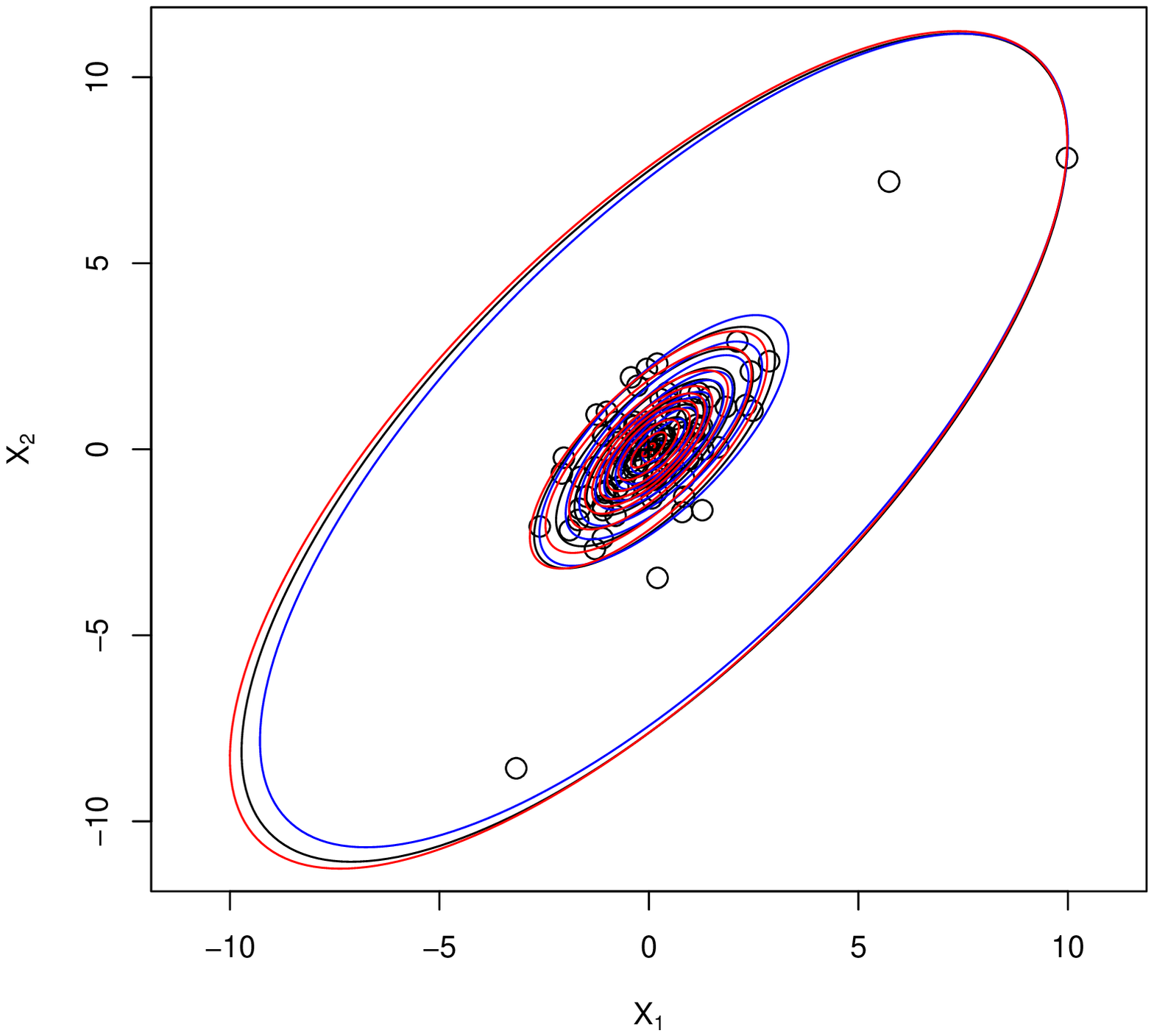} & 
\includegraphics[height=5.5cm]{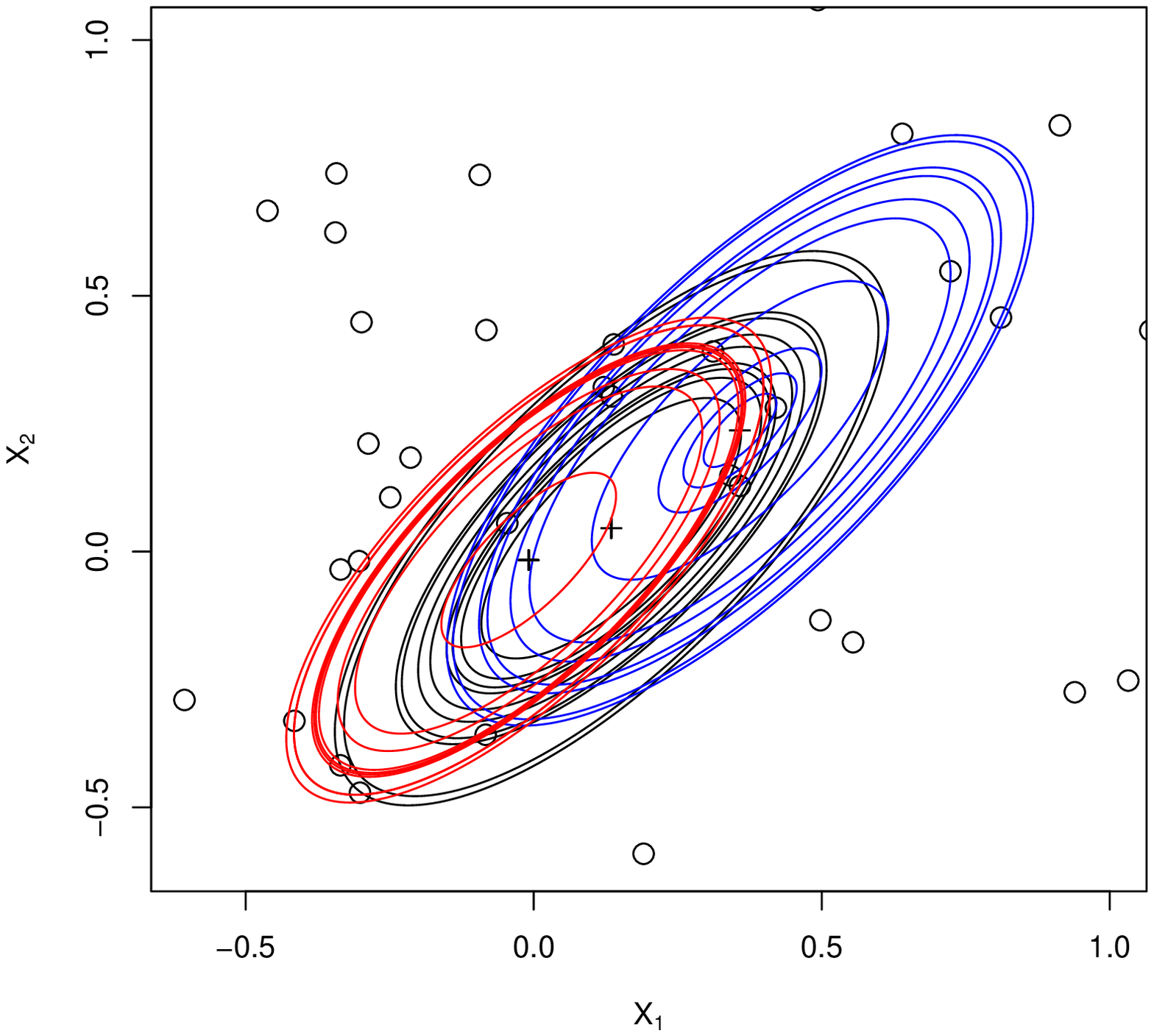}
\end{tabular}
\caption{$t$-distribution: Mahalanobis distances represented by co-axial ellipses centered at the mean $\bar{X}$ (black), at the  trimmed $\mathbf L_{n1}$-estimator (blue) and  at the $\mathbf L_{n2}$-estimator (red). All simulated bivariate data with the every tenth contour are illustrated on the left, detail of the center with the first ten contours is on the right.
}%
\label{figure12}
\end{center}
\end{figure}

\begin{acknowledgement}
We highly appreciate the long-term cooperation with Hannu Oja and his group, and hope in new fresh interesting joint ideas in the future. We also thank the Editors for the organization of this Volume.

The authors thank two Referees for their comments, which helped to better understanding the text. P. K. Sen gratefully acknowledges the support of the Cary C. Boshamer Research Foundation at UNC. Research of J. Jure\v{c}kov\'a and of J. Picek was supported by the Czech Republic Grant 15--00243S. 

\end{acknowledgement}


\end{document}